\documentclass[11pt]{article}
\usepackage{epsfig,psfrag,amsmath,amssymb,amsthm,bbm}
\usepackage{accents}
\usepackage{enumerate}
\usepackage{amsfonts,amsmath}
\usepackage{bbm}
\usepackage[dvipsnames]{xcolor}

\newcommand {\IZ}{\mathbb{Z}}
\newcommand {\IN}{\mathbb{N}}  
\newcommand {\IR}{\mathbb{R}}   

\newtheorem{stat}{Statement}

\newtheorem{prop}[stat]{Proposition}
\newtheorem{cor}[stat]{Corollary}
\newtheorem{thm}[stat]{Theorem}
\newtheorem{lemma}[stat]{Lemma}
\newtheorem{remark}[]{Remark}
\newtheorem{def1}[]{Definition}

\definecolor{mygreen}{rgb}{0,0.7,0.4}
\definecolor{mypurple}{rgb}{0.7,0,0.3}

\title{ {Contact process under renewals II}}

\author{Luiz Renato Fontes\footnote{Instituto de Matem\'atica e
		Estat\'\i stica. Universidade de S\~ao Paulo, SP, Brazil. E-mail:
		lrfontes@usp.br},
	Thomas S. Mountford\footnote{\'Ecole Polytechnique F\'ed\'erale de Lausanne,
		D\'epartement de Math\'ematiques,
		1015 Lausanne, Switzerland.
		Email: thomas.mountford@epfl.ch},\\ and
	Maria Eul\'alia Vares\footnote{Instituto de Matem\'atica. Universidade Federal
		do Rio de Janeiro, RJ, Brazil. Email: eulalia@im.ufrj.br}}

\begin{document}

\maketitle

\begin{abstract}
We continue the study of renewal contact processes initiated in a companion paper, where we showed that if the tail of the interarrival distribution $\mu$ is heavier than  $t^{-\alpha}$ for some $\alpha <1$ (plus auxiliary regularity conditions) then the critical value vanishes. In this paper we show that if $\mu$ has decreasing hazard rate and tail bounded by $t^{-\alpha}$ with $\alpha >1$, then the critical value is positive in the one-dimensional case. A more robust and much simpler argument shows that the critical value is positive in any dimension whenever the interarrival distribution has a finite second moment.


\bigskip


\medskip

\end{abstract}
%

\setcounter{equation}{0}
\section{Introduction}
\label{sec;1}

\noindent In this article we continue the study of renewal contact processes begun in the companion paper \cite{FMMV}, but whereas that article gave general conditions for the critical value to equal zero, here we consider conditions entailing the strict positivity of the critical value.

\vspace{0.2cm}

\noindent The renewal contact process is heuristically a model of infection spread, taking values in $\{0, 1\}^{\IZ^{d}}$, where for a configuration $\xi \in \{0, 1\}^{\IZ^{d}}$, the value $\xi(x) = 1$ indicates that individual $x$ is sick and $\xi (x) = 0$ means it is healthy.
\noindent Healthy individuals become sick at a rate equal to some fixed parameter $\lambda$ times the number of infected neighbours. Once sick, the sickness lasts until the next occurrence of a renewal process at the corresponding site; the renewal sequences are independent with the same interarrival distribution $\mu$ for all $x$.  Upon completion of this renewal period the individual reverts to the healthy (but reinfectable) state it had prior to this infection. When $\mu$ is the exponential distribution (typically fixed with rate 1), this is the classical Harris contact process. With general distributions for the interarrival times, we lose the Markov property, but it can still sensibly be viewed as having a percolation structure.  This work, as well as the companion paper \cite{FMMV}, has affinities with \cite{K} and \cite{NV}, which considered contact processes with exponential infections and transmissions but where the rates were randomly assigned.

\vspace{0.2cm}

The setup is the same as in \cite{FMMV}. We have for each ordered pair $(x,y)$ of neighbouring points in $\IZ^d$ (in the usual $\ell_1$-norm) a Poisson process $N_{x,y}$ of rate $\lambda$ (or a process $N_{x,y, \lambda}$ if one is interested in comparing processes with differing infection rates).  We also associate renewal processes $\mathcal{R}_x$
for $x \in \IZ^d$.  All these processes are independent of each other.  Typically but not always (see Section \ref{sec;3}) the $\mathcal{R}_x$ are taken to be i.i.d. renewal processes starting at $0$.  In this latter case we may write
$$
\mathcal{R}_x \quad = \quad \{S_{x,n}: n \geq 1\},
$$
where $S_{x,n} = \sum_{k=1}^n T_{x,k}$ for $ \{ T_{x,k}\colon x \in \IZ^d,k \ \geq 1\}$ i.i.d. random variables with
law the designated $\mu$.

Our process is then constructed via {\it paths}.   A path from $(x,s)$ to $(y,t)$ for $x,y \in \IZ^d$ and $s< t$ is a c\`adl\`ag function $\gamma: [s,t] \ \rightarrow \ \IZ^d$ so that \\
\indent
(i) $\gamma (s) \ = \ x$; \\
\indent
(ii) $\gamma (t) \ = \ y$
; \\
\indent
(iii) $\forall u \in [s,t], \quad u \ \notin  \mathcal{R}_{\gamma (u)}; $ \\
\indent
(iv) $\forall u \in [s,t], \quad $ if  $\gamma (u-) \ne \gamma (u) $, then $u \in  N_{\gamma (u-),\gamma (u)}$.
\vspace{0,2cm}

\noindent Except for Section \ref{sec;2} we will be dealing with $d=1$ in this paper.

\begin{def1}\label{crossing}
Given bounded subsets of $\IZ^d \times \mathbb R$,  $C$ and $D$, we say there is a crossing from
$C$ to $D$ if there exists a path $\gamma: [s,t] \ \rightarrow \ \IZ^d $ so that \\
$(\gamma (s),s ) \ \in \ C$ and $(\gamma (t),t ) \ \in \ D$.
\end{def1}

Given these processes, the renewal contact process (RCP)
starting at $A \subset \IZ^d , \ \xi ^A_t$ is, as usual, defined by
$$
\xi^A_t(y) \ = \ 1 \iff  \  \exists \mbox{ a path from } (x,0)    \mbox{ to } (y,t) \mbox{ for some }
x \ \in \ A.
$$
(If the infection rate is not fixed we may also write it as $\xi^{A, \lambda}_t$.)

For this process we have (taking the usual identification of $\xi : \IZ^d \rightarrow \{0,1\}$ with the subset of points in $\IZ^d$ with $\xi $ value $1$) that
$$
\xi^A_t \ = \ \cup _{x \in A}\xi ^{\{x\}}_t.
$$
That is, like the classical contact process, the process is additive.

Besides losing the Markov property (unless the law $\mu $ is exponential), we no longer typically have the FKG property, though (see Section \ref{sec;3}) there is a larger class of renewal processes for which this holds.

On the other hand, in our model the processes $N_{x,y,\lambda}$ remain independent Poisson processes and we may construct these processes so that
$$
\forall \lambda < \lambda', \ x, y \quad N_{x,y , \lambda}  \ \subset \ N_{x,y,\lambda'}.
$$
This being the case, if we use the same renewal processes to generate the respective contact processes, we have
$$
\forall A \ \subset \ \IZ, \ x , \ \ \lambda < \lambda' , \ \xi ^{A, \lambda}_t (x) \leq \ \xi ^{A, \lambda '}_t (x).
$$
From this we immediately have that $\exists \ \lambda _c \ \in \ [0, \infty ]$ so that \\
$ \lambda < \lambda _c$ implies $P(\xi^{\{0\},\lambda }_t = \emptyset $ for all large $t)$ = 1, and \\
$ \lambda > \lambda _c$ implies $P(\xi_t ^{\{0\}, \lambda} \ne  \emptyset $ for all  $t) \ >  \ 0$.

Equivalently,
$$
\lambda_c=\inf\{\lambda \colon P(\tau^{0}=\infty)>0\},
$$
where $\tau^{0} = \inf \{t \colon \xi^{\{0\}}_t=\emptyset \}$.

By additivity and translation invariance of the process, for any finite $A \ \subset \ \IZ^d$,
$ \lambda < \lambda _c$ implies $P(\xi ^{A, \lambda }_t =\emptyset $ for all large $t) \  = \ 1$ and
$ \lambda > \lambda _c$ implies $P(\xi ^{A, \lambda }_t \ne \emptyset $ for all  $t) \ >  \ 0$.

In general the value $\lambda_c $ need not be  strictly positive and indeed our first paper shows that in a large class of cases $\lambda_c $ is in fact $0$. In that paper we showed that if the law $\mu$ had the property that there exist $\epsilon, C_1 >0$ and $t_0 > 0$ so that $\mu ([t, \infty)) \geq C_1 /t^{1 - \epsilon}$ for all $t \geq t_0$, then (given auxiliary regularity hypotheses) our process had critical value $0$. Here we show that if the tails are suitably bounded then the critical value must be strictly positive when $d=1$.

\noindent We begin with the easiest case of finite second moment:

\begin{thm}
\label{thm1}
\noindent For a renewal contact process on $\IZ^{d}$, if the law $\mu$ satisfies $\int t^{2} \mu(dt) < \infty$ then $\lambda_{c} > 0.$
\end{thm}

\vspace{0.3cm}

\noindent This argument uses a branching process argument which is somewhat hidden by the given non Markov renewal structure. We would like to emphasize that this result requires no auxiliary regularity assumptions and is valid in all dimensions. Indeed it is valid in the more general framework of graphs of bounded degree.  Furthermore if we recast the question as a percolation problem
where space time point $(x,t) \ \in \ \IZ^d \times \IR_+ $ is connected to space time $(y,s) $ if there exists $n$ and $\{x_i\}_{i=0} ^ n , \ \{t_i\}_{i=0} ^ n $ so that \\
\indent
(i) $x_0 \ = x, \ t_0 \ = \ t $ and $x_ {n-1} \ = \ x_n \ = y, \ t_n \ = \ s $, \\
\indent
(ii) $\forall \ 0 \leq i \ < \ n-1, \quad |x_i - x_{i+1}| = 1 $ and $\forall \ 1 \leq i \ < \ n-1, \quad t_i \ \in \ N_{x_i, x_{i+1}} \\$
and \\
\indent
(iii)  $\forall \ 0 \leq i \ < \ n, \quad   \mathcal{R}_{x_i}\cap [t_i, t_{i+1}] \ = \ \emptyset$,\\
then the given argument shows that (in the obvious sense) there is no percolation for small $\lambda$.

The argument  leaves a definite gap with the previous results: ignoring technical assumptions, if the tail $\mu ([t, \infty))$ is ``like" $\frac{1}{t^{1 - \epsilon}}$ then $\lambda_{c} = 0$, if it is ``like" $\frac{1}{t^{2 + \epsilon}}$ then $\lambda_{c} > 0$.

\vspace{0.3cm}

\noindent The next theorem is the main result of the paper and makes a step in the direction of filling this gap. It reverts to classical percolation ideas such as RSW crossing estimates and a recursion argument to push these together. It also requires the use of FKG inequalities, which imposes more stringent assumptions on $\mu$:

\vspace{0.3cm}

\noindent {\bf Hypothesis A:}
\noindent $\mu$ has a density $f$ and distribution function $F(t) \ = \ \int_0 ^t f(u)du $ so that the hazard rate $\frac{f(t)}{1 - F(t)}$ is decreasing in $t$.


\begin{thm}
\label{theo2}
\noindent Let $\mu$ satisfy hypothesis A above and $\int t^{\alpha} \mu(dt) < \infty$ for some $\alpha >1$. Then the corresponding renewal contact process on $\IZ$ has strictly positive critical value.
\end{thm}


\noindent {\bf Remark.} The arguments used in the proof of Theorem \ref{theo2} rely on putting together distinct crossing paths, which means that our proof works only for $d=1$.

\medskip

\noindent {\it Outline of the proof.}
Let us at this point give an overall picture of our strategy to prove Theorem~\ref{theo2}. There are three main parts. First, we relate the survival of the infection from the origin up to time $2^n$ to space or time crossings (to be precisely defined in Section 4) of space-time rectangles of spatial and temporal side lengths  $\lfloor 2^{r\beta}\rfloor$ and $2^{r}$, respectively, for suitable $\beta\in(0,1)$ and $r\leq n$. See proof of Theorem~\ref{theo2} (at the beginning of Section~\ref{sec;5}) below. In this part, dimensionality and the FKG inequality play a crucial role.
%


From the first part, it is enough to show that the probability of the space or time crossings mentioned above vanishes as $r\to \infty$. This is the content of Proposition~\ref{propkey}, which is in turn proved via a recursion scheme, in two more parts, as follows. Let us focus on time crossings (the space crossings are treated similarly, if more simply). A time crossing of $[0,\lfloor 2^{n\beta}\rfloor]\times[0,2^{n}]$ implies the time crossings of $2^k$ subrectangles $[0, \lfloor 2^{n\beta}\rfloor]\times[i 2^{n-k}, (i+1)2^{n-k}]$. Here $k$ is a fixed (large) number, independent of $n$. We need to estimate the successive conditional probabilities. Since we have a renewal process on each time-line $\{x\} \times [0, \infty)$, in the event, say $A$, that for each even $i$ and $x\in [0,\lfloor 2^{n\beta}\rfloor]$ there is a renewal mark in the previous time interval $[(i-1)2^{n-k}, i2^{n-k}]$, we get that the conditional probability of a crossing of the  $i$-th subrectangle, given the first renewal marks in the previous subrectangle and all previous history, becomes independent of the history up to the previous even rectangle; a product of the (sups of) crossing probabilities (with the renewal processes starting from different points in the previous subrectangle) over the even subrectangles ensues. The probability of the complement of the above mentioned event $A$ is controlled by the integrability assumption on $\mu$.
Yet, the subrectangles do not have the proper $\lfloor 2^{\ell\beta}\rfloor\times2^\ell$ dimensions.   We relate each of these events to space or time crossings of rectangles of dimensions $\lfloor 2^{\beta (n-\ell)}\rfloor \times2^{n-\ell}$, with $\ell=k$ or $\ell =k+1$. This involves considering a number of cases where such crossings take place, as done in Subsection \ref{ssec;5.1}. In most cases it is just a matter of dealing with a union bound (depending on the location of the crossing). Nevertheless, there is one case where we need again to use Lemma \ref{build}, where FKG is crucial. This is the second part, accomplished in Proposition~\ref{xiv}.

In the concluding argument we use the second part to set up a $k$-step recursion scheme, see~(\ref{recur}),  by the iteration of which, using the decay of the distribution of the inter-arrival times and taking $\lambda$ small, we get the final result.


\section{Finite Second Moment. Proof of Theorem \ref{thm1}}
\label{sec;2}

In this section we assume that $\int t^2 \mu(dt)<\infty$.  The importance of this hypothesis is that it yields the following property for our renewal process $\mathcal{R}$ upon which the proof relies: \\
\indent
There exists $C < \ \infty $ so that uniformly over $t \geq 0  $ the length of the renewal interval $I_t$ which contains the point $t $ satisfies $E(|I_t|) < C$. ($\ast$)

\medskip


A key part of the analysis is to consider ``intervals" infected by the origin $(0,0)$.  More precisely, an ``infected interval" is a subset of $\{x \} \times \IR_+$ of the form   $\{x \} \times J $ for some $x \ \in \ \IZ^d $  and some interval $J \subset \IR$ so that all points
$(x,t)$ in it satisfy $(0,0) \rightarrow (x,t) $ (and no points in it belong to $\mathcal{R}$) and finally it is a maximal subset with this property.  So an infected interval, $I$, will be of the form
$\{x \} \times [s_I,t_I)$ where $s_I$ is its infection time and $t_I $ is the first time point after $s_I$ that belongs to $\mathcal{R}_x$.

We now introduce a ``coding" of infected intervals.  The interval containing $(0,0)$ is coded as $\emptyset $.  Other infected intervals are coded recursively.  If $I \ = \ \{x \} \times [s_I,t_I) $ and
$s_I \ \in \ N_{y,x}$, then for some positive integers $k$ and $i_j $, $1 \leq j \leq k$, we code $I$ by $(i_1, \cdots i_k)$ if $(y,s_I) $ belonged to an interval coded $(i_1, \cdots i_{k-1})$ and if $s_I$ is the $i_k$'th infection point (in chronological order) in the interval coded $(i_1, \cdots i_{k-1})$.  We can think of $k$ as the ``generation" of interval $I$.  We stress that the generation corresponds to the first infection time and not to the ``smallest possible" $k$.
Thus not all arrows result in the creation of an infected interval.  If the r-th arrow of interval  $(i_1, \cdots i_{k-1})$ (here we identify intervals and their codes) infects an already infected site, then the interval $(i_1, \cdots i_{k-1},r)$ is empty or nonexistent (or the arrow is wasted).

Next we define $\IZ_+$ valued random variables $X_{\underline i } $ for $\underline i  \in \cup _{k=0}^ \infty \IN^k$,
with $\IN^0$ denoting the code $\emptyset$, so that
$X_{\underline i } $
equals the number of arrows to neighbouring time lines for interval
$\underline i$.
This will  naturally equal zero if ``interval" $\underline i$
is empty.
We note the branching process property of the $X_{\underline i}$'s:
\begin{equation}\label{branch}
X_{(i_1, \cdots i_{k-1})} \ = \ 0 \quad \Rightarrow \ X_{(i_1, \cdots i_k)} \ = \ 0,\ i_k\geq1.
\end{equation}
It follows that if, for some fixed $k ,  \displaystyle \!\!\!\sum _{(i_1, \cdots i_{k})}\!\!\! X_{(i_1, \cdots i_k) }  =  0 $, then $ \displaystyle\!\!\!\sum _{(i_1, \cdots i_{k'})}\!\!\! X_{(i_1, \cdots i_{k'}) }  =   0 $
for each $k' > k $, and there are only finitely many infected intervals.  This will immediately imply that the contact process dies out.


In fact, we can go beyond (\ref{branch}) to say that
$ \sigma_{(i_1, \cdots i_k)}  = \infty $ implies that $ X_{(i_1, \cdots i_k)} =0$, where $ \sigma_{(i_1, \cdots i_k)} $ is the time of the $i_k $'th arrow of interval $ (i_1, \cdots i_{k-1}) $.

Property ($\ast$) at the beginning of the section implies that
$$
E(X_{(i_1, \cdots i_k)} |  \sigma_{(i_1, \cdots i_k)} < \infty ) \ \leq \ 2Cd \lambda .
$$
From this we inductively get that $E \left( \sum _{(i_1, \cdots i_{k})} X_{(i_1, \cdots i_k)}  \right)  \leq  (2Cd \lambda ) ^{k+1}$.  The condition $\lambda < 1/2Cd$ thus implies that a.s.~the
contact process dies out, concluding the proof of Theorem \ref{thm1}.

\section{Hypothesis A and FKG inequalities}
\label{sec;3}

\noindent This section clarifies the role of Hypothesis A. As stated in Proposition \ref{propfkg} below, it guarantees the FKG property for our RCP, which will then be important for the estimates for crossing probabilities developed in the next section, and which lead to the proof of the main theorem.

We shall deal with a family of independent renewal processes, starting from possibly different initial points. Let $f$ be a probability density on $\mathbb{R}_+$ and $F$ the corresponding distribution function.  We assume that Hypothesis A is satisfied. A realization of the corresponding renewal process starting at any point $t_0 \in \mathbb{R}$ can be easily obtained in terms of a homogeneous Poisson point process $\eta$  on $\mathbb{R} \times \mathbb{R}_+$ of intensity 1.


For this let $h$ be the hazard rate function, defined as $h(t)=f(t)/(1-F(t))$.  We note that under hypothesis A, $F(t) \ \in \ (0,1)$ for all $t > 0 $. To construct the renewal process starting at some point $t_0 \in \mathbb{R}$ we consider all points of $\eta$ in $(t_0,\infty) \times (0,\infty)$ that are under the graph of the function $t \mapsto h(t-t_0)$. Since $\int_0^t h(s)ds =-\log (1-F(t))$, with probability one there are infinitely many such points but only a finite number with first coordinate in $[t_0,t_0+t]$ whenever $F(t)<1$. We can then take the point with the smallest first coordinate, call it $(t_1,u_1)$, i.e.~$u_1 \le h(t_1-t_0)$ and there is no point $(s,u)$ in $\eta$ with $u \le h(s-t_0)$ and $t_0 < s<t_1$. We then have $P(t_1-t_0>s)=e^{-\int_{0}^s h(v)dv}=1-F(s)$ i.e. $t_1-t_0$ has the renewal distribution $F$.  Having obtained $t_1$ we repeat the procedure replacing $t_0$ by $t_1$, since of course the variable $t_1$ is a stopping time for the filtration $(\mathcal{F}_s)_s$ generated by $\eta$ restricted to  $[t_0,\infty)\times (0,\infty)$, i.e.  $\mathcal{F}_s=\sigma(\eta(B)\colon B \subset [t_0, s] \times (0,\infty), B\text{ Borel})$. In this way, and using the independence property of the Poisson variables $\eta(B)$ for disjoint Borel sets $B$, we get $t_1<t_2<\dots $ so that $t_i-t_{i-1}, i \ge 1$ are i.i.d. with density $f$.

%
%

\vspace{0.3cm}
\smallskip

	
	For the FKG property, the important point to realize is that, due to the assumption of decreasing hazard rate, the renewal process is an increasing function of points in the Poisson point process; if a P.p.p realization $\eta'$ differs from $\eta$ by the addition of a point $(s,u)$, then either $u$ is insufficiently small to add $s$ to the renewal set $\mathcal{R}$ and nothing changes, or $s$ is added. In this case, we need to see that the sequence corresponding to $\eta'$ contains that of $\eta$. Let us write $t_1<t_2<t_3<\dots$ for the sequence $\mathcal {R}$ corresponding to $\eta$ and let us assume $t_{j} < s < t_{j+1}$. It is obvious that nothing changes up to $t_j$. When $s$ is added, i.e. we have $u\le h(s-t_j)$, we observe that the next point in $\mathcal{R}'$ will be obtained by checking the $\eta$ points that are under the graph of $v \in (s,\infty) \mapsto h(v-s)$, and taking the one with smallest first coordinate. Since $s \ge t_j$  we have $h(v-s)\ge h(v-t_j)$ for all $v\ge s$, so that
$t_{j+1}$ is one of such points, but there could be one with smaller first coordinate $t'_{j}$. In this case $t'_{j}$ is added to the sequence $\mathcal{R}'$ and we repeat the argument with $t'_j$ instead of $s$. It is easy to see that after a finite number of extra points less than $t_{j+1}$ we shall add $t_{j+1}$ and from that point on, the sequences continue in the same manner.


\vspace{0.3cm}

\noindent  We now consider an event depending on a finite space time rectangle $[0, L] \times [0, T]$  of renewal points $D_x=\{(x,S_{x,n})\}$ and $\lambda$ Poisson processes $\{N_{x,y}\}$ of arrows. We can and will assume that the renewal times $\mathcal{R}_x =\{S_{x,n}\}$ for $x \in [0,L]$, are generated by independent Poisson point processes $\eta_x$ as just discussed.

\vspace{0.3cm}
\noindent {\bf Definition.}  (i) An event $A$ is said to be {\it increasing} with respect to the $\lambda$ Poisson processes $\{N_{x,y}\}$ if given any joint realizations $\omega$ and $\omega'$ of the renewal sequences and $\lambda$ Poisson processes
such that $\omega$ and $\omega'$ have the same renewal points and the $\lambda$  Poisson points in $\omega$ are also present in $\omega'$, then $\omega \in A$ implies $\omega' \in A$.

%

\noindent (ii) An event is {\it decreasing} with respect to the renewal processes if whenever the configurations $\omega$ and $\omega '$ have the same $\lambda$ Poisson process realizations and the renewal processes of $\omega $ dominate those of $\omega '$ (in the sense that if for some $x \in [0,L]$, $(t,u') \in \eta_x (\omega' ) $,  then $(t,u)\in \eta_x (\omega ) $ for some $u \leq u'$), then $\omega \in A$ implies $\omega' \ \in \ A$.

\noindent (iii) We say that an event depending on renewal and $\lambda$ Poisson process points in a finite space time rectangle is {\it increasing} if it is increasing with respect to the $\lambda$ Poisson processes of arrows, and decreasing with the renewal processes.

\vspace{0,3cm}

We then have, by the previous observations (and usual discretization arguments), the following FKG inequality:

\begin{prop}
\label{propfkg}
Assume that the renewal sequence satisfies hypothesis A, and let $A_1, A_2, \dots, A_n $ be increasing events on a finite space time rectangle. Then
\begin{equation*}
P(\cap_{i=1} ^ n A_i ) \ \geq \ \prod _{i=1} ^ n P(A_i).
\end{equation*}
\end{prop}

\medskip

\begin{remark}
\label{comparison}
Let $\mathcal R$ and $\tilde{\mathcal{R}}$ be renewal processes starting at $0$ and at some $t_0>0$, respectively. If the interarrival distribution $\mu$ satisfies Hypothesis A, these processes may be coupled in such a way that the set of renewal marks of $\mathcal{R}$ that fall in $[t_0,\infty)$ is contained in the set of renewal marks of $\tilde{\mathcal R}$.
\end{remark}

\section{Applications of FKG inequalities to crossings}
\label{sec;4}
In this section we apply the previous result to a specific kind of crossing event of a rectangle, requiring the existence of a {\it sufficiently inclined diagonal path within a rectangle of certain dimensions} --- see~(\ref{A0}) below. This will be an important ingredient in our strategy of proof of Theorem~\ref{theo2}, as outlined at the end of the Introduction,  and to be undertaken in the following section.
See Lemma~\ref{build}, Corollary~\ref{FKG2} and Remark~\ref{fkg} below.

We are interested in the increasing events defined by crossings as in Definition \ref{crossing}.
%
%
%
%
%

\vspace{0,3cm}

\begin{def1}
  \label{crossing2}
We say there is a crossing from $C \ \subset \ \IZ \times \IR$ to $D\ \subset \ \IZ \times \IR$ {\it in} space-time region $H\ \subset \ \IZ \times \IR$ if there exists a path $\gamma : [s,t] \ \rightarrow \ \IZ $ as in Definition \ref{crossing} such that \\
(i) $( \gamma(s),s)  \ \in \ C$, \\
(i) $( \gamma(t),t)  \ \in \ D$, \\
and \\
(iii) for all $u \in [s,t], \ (\gamma(u),u) \ \in \ H$.
\end{def1}
Obviously the existence of a crossing is an increasing event no matter what choice of $C$, $D$ and $H$ is made. The definition above includes the following special cases:
\begin{def1}
\label{crossing3}
Given a space-time rectangle $H=[a,b] \times [S,T], \quad a,b \in \IZ, \ S,T \in \IR$, we say:

\noindent (I) $H$ has a {\it spatial} crossing if
there exists a crossing in $H$  from $C \ = \ \{a\}\times [S,T]$ to $D  =  \{b\} \times\ [S,T]$.

\noindent (II) $H$ has a {\it temporal} crossing if
there exists a crossing in $H$ from $C \ = \ [a,b] \times \{S\}$ to $D  =   [a,b] \times \{T\}$.
\end{def1}


\noindent {\bf Remark.} A space interval $[a,b]$ should be always understood as $[a,b]_\mathbb Z:=[a,b]\cap\mathbb Z$.

\vspace{0.3cm}

A useful ``building block" in analyzing spatial or temporal crossings of space-time rectangles is the event
\begin{equation}\label{A0}
A_0 \equiv \ A_0(c,\epsilon, L,T) 	
\end{equation}
for $T \in \IR_+, \ L \in \IZ_+ , \ 1/2 <c<1$ and $ \epsilon < cT/8.$
\noindent
$A_0 $ is the event that 
 there is a  crossing (in $[0,L] \times \mathbb{R}_+$) from $\{0\} \times [0, \epsilon]$ to $\{L\} \times [cT, cT + \epsilon]$.
\vspace{0.3cm}

\vspace{0.3cm}

\noindent Let $I_0=[0,\epsilon]$ and recursively define the time intervals $I_1=[cT,cT+\epsilon]$, $I_{2k}=I_{2k-1}-\epsilon$, $I_{2k+1}=I_{2k}+cT$, where $I\pm a=\{x \pm a\colon x \in I\}$. Define
$A_{1}$ to be the event that there is a crossing (within $[0, L]\times \IR_{+}$) from  $\{L\}\times I_{2}$ to  $\{0\} \times I_{3}$, and $A_{2}$ to be the event that there is a crossing (within $[0, L]\times \IR_{+}$) from  $\{0\}\times I_{4}$ to  $\{L\} \times I_{5}$, $A_{3}$ to be the event that there is a crossing (within $[0, L]\times \IR_{+}$) from  $\{L\}\times I_{6}$ to  $\{0\} \times I_{7}$, and so on.

%
%

\begin{lemma} \label{build}
$P (A_{0} \cap A_{1} \cdots \cap A_{m}) \geq  \prod_{i=0}^m P(A_i) \ge P (A_{0})^{m+1}$.
\end{lemma}
\begin{proof}
The first inequality follows from Proposition \ref{propfkg} since the events in question are increasing.  For the second inequality, observe that  for all $i, \ P(A_i) \geq P(A_0) $ by our choice of FKG renewal distribution, as follows from Remark \ref{comparison}.
\end{proof}
%

%

\begin{cor} \label{FKG2}
	\noindent Let $m$ be a positive integer. The probability of a temporal crossing of $[0, L] \times [\epsilon, \epsilon + mT]$ is at least $P(A_{0})^{\frac{8}{3}m +2}$.
\end{cor}
\begin{proof}
The rectangle in $A_i$, $i\geq0$, starts at time $i(cT-\epsilon)$ and has length $cT+\epsilon$ (in the temporal direction). It follows from the definitions that the event in the statement occurs in $A_{0}\cap \cdots\cap A_{n}$ provided $n(cT-\epsilon)+ cT \ge \epsilon + mT$. Therefore it suffices $n\geq\frac mx-1$, where $x=c-\epsilon/T$. From our hypotheses, we have that $x\in[\frac7{16},1)$, so the least integer $n$ satisfying the above condition is bounded above by $\lceil\frac{16}7m\rceil\leq\frac{8}{3}m +1$, and the result follows from Lemma~\ref{build}.
\end{proof}

\begin{remark}\label{fkg}
Given the FKG property of our renewal processes, the above bound holds if for $v \geq 0$ the event that there is a temporal crossing of $[0, L] \times [\epsilon, \epsilon + mT] $
is replaced by the event that  there is a temporal  crossing of $[0, L] \times [v+ \epsilon, v+ \epsilon + mT]  $
 in  space-time rectangle $[0, L]\times [v, \infty )$ and event $A_0 $ is replaced by the event that
there is a  crossing (in $[0,L] \times [v,\infty)$) from $\{0\} \times [v, v + \epsilon]$ to $\{L\} \times [v+ cT, v+ cT + \epsilon]$.


\end{remark}

\begin{remark}\label{value}
The value of this result is that if $c$ is not too small, then a reasonable probability for a  spatial crossing (in $[0, L]\times \IR_{+}$) from $\{0\} \times [0, \epsilon]$ to $ \{L\} \times [cT, cT + \epsilon]$
yields a not too small probability for a temporal crossing of rectangle $ [0, L]\times [0, mT]$.  Furthermore it is easy to see that if there is a reasonable probability for a
spatial crossing of $ [0, L]\times [0, T]$, then either there is a reasonable probability of a spatial crossing for which
the  time difference between its initial and final points is small (compared to $T$) or  a not too small probability of a temporal crossing of $[0, L]\times [0, mT]$ is entailed. This will be developed in the next section. 
\end{remark}


\section{Crossings of rectangles}
\label{sec;5}

In this section we prove Theorem~\ref{theo2}. We start the argument, following the first step of the strategy outlined at the end of the Introduction, by reducing survival to crossings of space-time rectangles.

\medskip

\noindent {\bf Notation.} For $x >0$, write $\lfloor x\rfloor =\max\{n \in \mathbb{Z}\colon n \le x\}$ and $\lceil x\rceil=\min \{ n \in \mathbb{Z} \colon n>x\}$.



\begin{def1}\label{pr}
Let $\beta\in (0,1)$. Here and in the following $P_r$ denotes the supremum over the  probabilities for the space-time rectangle $ [0, \lfloor 2^{r\beta}\rfloor ]\times [0,2^r ]$ of either a spatial or a temporal crossing. The supremum is taken over all product renewal probability measures with interarrival distribution $\mu$, for the death points starting at time points strictly less than $0$.  (Note the starting points (or times) need not be the same.)
\end{def1}



Due to the  FKG property,  $P_r$ is indeed the limit of the probability of crossings (either spatial or temporal) for $ [0, \lfloor 2^{r\beta}\rfloor ]\times [T,T+2^r ]$,  as $T$ tends to infinity.

We now state the key result for this section.   

\begin{prop} \label{propkey}
Assume $\beta\in (0,\alpha-1)$, with $\alpha$ as in the statement of Theorem~\ref{theo2}. There exists $\lambda_0  >   0$ so that for $0 \leq  \lambda  <  \lambda_0$ \\
$$
P_r \ \stackrel {r \rightarrow \infty }{\longrightarrow } \ 0.
$$
\end{prop}

Given this result and Lemma~\ref{build}, we quickly achieve our desired result:

\medskip

We now give the proof of Theorem~\ref{theo2}.\\

\begin{proof}
It is enough to show that $P ( \tau^0 = \infty )  =  0$ for $\lambda <  \lambda_0 $, the claimed constant of Proposition \ref{propkey}.
Equivalently we must show that
$P ( \tau^0 >  2 ^ r )$
tends to zero as $r $ tends to infinity.

Consider the event that $\tau^0 >  2^r$.  This is contained in the union of three events defined by
the Harris system on space-time rectangle $R =[-\lfloor 2^{r\beta} / 2\rfloor, \lfloor 2^{r\beta }/ 2\rfloor]\times [0, 2^r ]$: \\
(I) there exists a path from $(0,0)$ to $ \mathbb Z\times \{ 2^r \}$ in $R$; \\
(II) there exists a path from $(0,0)$ to $\{ \lfloor 2^{r\beta}/ 2\rfloor \} \times [0, 2^r ]$; \\
(III) there exists a path from $(0,0)$ to $\{ -\lfloor 2^{r\beta} / 2 \rfloor \} \times [0, 2^r ] $.\\

The first possibility (I) is simply a subset of the event that the space-time rectangle $R$ has a temporal crossing and so (given translation invariance of the system) has probability bounded by $P_r$ which, by Proposition \ref{propkey}, tends to zero as $r$ tends to infinity.  So it remains to find an upper bound for possibilities (II) and (III) which tends to zero as $r$ tends to infinity.  By symmetry we need only upper bound the
probability of the event (II).

We fix a large integer $K$ which will depend upon $\beta $ but not upon $r$, and for $1\leq i\leq j \leq K $ we define $A(i,j)$
as the event that there is a crossing from
$\{ 0 \} \times [\frac{i-1}{K} 2^r,  \frac{i}{K} 2^r]$ to $ \{ \lfloor2^{r\beta } / 2 \rfloor \}\times [\frac{j-1}{K} 2^r,  \frac{j}{K} 2^r] $ in the rectangle $R^\prime  =[0, \lfloor 2^{r\beta }/ 2\rfloor]\times [0, 2^r ]$.

Obviously the event $\cup_{i,j} A(i,j)$ contains (II).

We fix $1 \leq i \leq K $ and then split $i \leq j \leq K $ into
$B_i \ = \ \{ j: \frac{j-i+1}{K}  \ \leq \ 2^{-\lceil 1/ \beta\rceil}\}$
and $D_i = [i,  K] \backslash B_i$.

We have that the event $\cup_{j \in B_i}A(i,j)$ is contained in the event that there is a spatial crossing for the space-time rectangle $[0,\lfloor 2^{(r-\lceil 1/ \beta\rceil) \beta}\rfloor ] \times [\frac{i-1}{K} 2^r,  \frac{i-1}{K} 2^r+ 2^{r-\lceil 1/ \beta\rceil }] $ and so its probability is bounded by $P_{r - \lceil 1/ \beta\rceil}$.

For $j  \in  D_i $ we have (assuming that $K$ is sufficiently large) that  $ \epsilon \ = \ 2^r / K ,  cT = (j-i) 2^r / K $ and $ c =  2/3$ satisfy
$ 1/2 < c < 1 $ and $ \epsilon < cT/8 $ . So by Lemma~\ref{build} (and Remark \ref{fkg}) we have again assuming $K$ was fixed large)
$$
P(A(i,j)) \ \leq \ (P_r)^{1/( 2^{\lceil 1/\beta\rceil+1})}.
$$
Thus we obtain the bound for the probability of event (II)
$$
K^2  (P_r) ^{1/ (2^{\lceil 1/ \beta\rceil+1})} \ + \ K P_{r - \lceil 1/ \beta\rceil}
$$
and, again by Proposition \ref{propkey}, we are done.
\end{proof}

\subsection{Proof of Proposition~\ref{propkey} --- Generic crossing events}
\label{ssec;5.1}




We start by introducing some generic crossing events which come up in different kinds of spatial or temporal crossings entering our analysis of $P_r$, as already anticipated, and deriving probability bounds for each of them.

\vspace{0.2cm}

\noindent {\bf Notation.} If $X=(X(u)\colon u \in [s,t])$  is a path, we write
\begin{equation}
\label{var}
v(X):=\max\{X(u)\colon u\in [s,t]\}-\min\{X(u)\colon u\in [s,t]\},
\end{equation}
and call $v(X)$ {\it variation} of $X$.

\noindent {\bf Definition.}
	For $D = [a, b]\times [s', t']$ a space-time rectangle, $c \in (0,1)$ a constant, and $r$ an integer, let $A(D, c, r)$ be the event that either there exist times $s_{1}, s_{2}$ with $s_{2} - s_{1} >  2^{r}/c$ and a path $X=(X(s)\colon s_{1} \leq s \leq s_{2})$ within $D$ such that  $v(X) < \lfloor c 2^{r \beta}\rfloor$, or there exist times $s_{1} \le  s_{2}$ with $s_{2} - s_{1} < c2^r $ and a path $X=(X(s)\colon s_{1} \leq s \leq s_{2})$ within $D$ so that $v(X) > \lceil\frac{2^{r \beta}}{c}\rceil$.

We then have

\begin{prop}\label{adcr}
For $D$, $c$ and $r$ as above with $2^{r \beta } (1-c) \ \geq \ 2 $,
$$ P(A (D, c, r)) \leq C(c) \left( \frac{b-a}{2^{r\beta}} \vee 1 \right) \left( \frac{t' -s'}{2^{r}}  \vee 1 \right) P_{r},$$
where $C(c)$ is a finite function.
\end{prop}

\begin{proof}
The proof consists of upper bounding the probabilities for either spatial or temporal crossings of rectangles.  It is sufficient to do the bounds separately.  We will do the bound
for the first case ($s_2 -s_1 > 2^r / c$) only since the proof for the other case is much the same.  Suppose there exists a path $X:[s_1,s_2] \rightarrow [a,b] $ (i.e. it satisfies conditions (i)-(iv) just before Definition \ref{crossing}) so that
$s_2 - s_1  > 2^r/c $ and $\max X(u) - \min X(u) <   \lfloor c2^{r \beta }  \rfloor $.  Let $s_1 ' = \inf \{s \geq s_1, \ s \in s' + \frac{1 - c}{2}2^r  \mathbb {Z} \} $, then the path
$X$ restricted to interval $[s_1', s_1' + 2^r]$ is a temporal crossing of the space-time rectangle $[a,b]\times [s_1', s_1' + 2^r]$ whose variation is less than $\lfloor c 2^{r\beta}  \rfloor $.
We then have that $X([s_1', s_1' + 2^r]) \ \subset \ [x_1', x_1'+ 2^{r \beta } ] $, where
$$
x_1' =\sup \{ x \in  a + \lfloor \frac{1 - c}{2}2^{r \beta } \rfloor \mathbb {Z} : \  x \leq \inf_{u \in [s_1', s_1' + 2^r] } X(u)\}.
$$

From this we see that the existence of $s_{1}, s_{2}$ with $s_{2} - s_{1} >  2^{r}/c$ and a path $X=(X(s) \colon s_{1} \leq s \leq s_{2})$ contained in $D$,  $v(X) < \lfloor c2^{r \beta}\rfloor$ implies the occurrence of $\cup _ {i,j} A_t(i,j,c) $, where $i,j$ range over the set of integers so that
$ (a(i,\beta), s'(j,\beta)):=( a +  i \lfloor \frac{1 - c}{2}2^{r \beta} \rfloor  ,s' + j\frac{1 - c}{2}2^r ) \in D$ and $ A_{t}(i, j, c)$ denotes the event that the space-time rectangle
$$
\left[a(i,\beta),a(i,\beta) +  2^{r \beta}  \right]\times \left[s'(j,\beta),s'(j,\beta) + 2^{r}  \right].
 $$
%
has a temporal crossing. 

The proof is completed by computing the simple upper bound for the number of such $(i,j)$, that is the number of $i $ so that $ a \ \leq \ a +  i \lfloor \frac{1 - c}{2}2^{r \beta} \rfloor \ < \ b$
and $j$ so that $ s' \ \leq \ s' +  j \frac{1 - c}{2}2^{r }  \ < \ t'$.  The latter number is bounded by the least integer superior to $2(t'-s') /(1-c)2^r \ \leq \ \frac{4}{1-c} \left( \frac{t'-s'}{2^r} \vee 1\right)$, while the former is bounded by the least integer superior to $\frac{(b-a)}{\lfloor (1 - c)2^{r \beta}/2 \rfloor }  \leq 1+ 2 \frac{(b-a)}{ (1 - c)2^{r \beta}/2  }$ by our assumption that $ (1 - c)2^{r \beta}/2 \geq 1$.  This in turn is bounded above by $\frac{4}{1-c} \left( \frac{b-a}{2^{r \beta} } \vee 1 \right) $.

\end{proof}

\noindent {\bf Definition.}
\noindent For a spatial interval $I$, $t\ge 0, r \ge 0$ and $c \in (0,1)$, let $B_{t} (c, I, r)$ denote the event that there exists a spatial interval $I' \subset I$ of length less than $\lfloor c 2^{r \beta}\rfloor$ so that there is a temporal crossing of
$I'\times [t, t + 2^{r} ]$.


Then we have:

\begin{lemma} \label{lem4}
Suppose that $2^{r \beta } (1-c) \ \geq \ 2 $, then
$P ( B_{t} (c, I, r )) \leq  C (c)( \frac{|I|}{2^{r \beta}} \vee 1 )  P_r$ for some finite $C (c)$ which depends on $c$ only.
\end{lemma}
\begin{proof}
This follows in similar fashion to the previous result.  Let $I \ = \ [a,b] $ and as above let $ a(i,\beta) \ = \ a +  i \lfloor \frac{1 - c}{2}2^{r \beta} \rfloor $
for $0 \leq i \leq \frac{b-a}{\lfloor 2^{r \beta } (1-c)/2 \rfloor } $.  Then every spatial interval, $J'$, of length at most $\lfloor c 2^{r \beta} \rfloor$ which is a subset of $I$ is contained in
an interval $[ a(i,\beta), a(i,\beta)+ \lfloor 2^{r \beta } \rfloor ] $ for some $0 \leq i \leq \frac{b-a}{\lfloor 2^{r \beta } (1-c)/2 \rfloor }  $.  As before under the condition   $2^{r \beta } (1-c) \ \geq \ 2 $, the number of such $i$ is less than $\frac{4}{1-c} \left( \frac{b-a}{2^{r \beta} } \vee 1 \right) $ and the result follows.
\end{proof}
\vspace{0.3cm}
Similarly we have,
\begin{lemma}  \label{lemad2}
For a space-time rectangle
$R \ = \ [a,b]\times [s,s+ 2^r]$, $k\in(0,r]\cap\mathbb Z$  and $c \in (0,1)$, let $W(R,r, c,k) $ be the event that there exists a spatial crossing of a rectangle $I \times[s,s+ 2^r] \subset R $, where interval $I$ has length at least $2^{r \beta}/c$. \\
We suppose that $b-a > 2^{(r-k) \beta}$. Then
$$
P(W(R,r,c,k)) \ \leq \ K(c)  \frac{b-a}{2^{(r-k) \beta}} P_r
$$
for suitable $K(c)$ finite.
\end{lemma}

And similarly we have:
\begin{lemma} \label{lemad1}
For a space-time rectangle
$R \ = \ [a,b]\times [s,t]$, where $b-a \geq 2^{(r-k^*) \beta } $ and $ t-s \geq \ 2^r  , \ k^*\in(0,r]\cap\mathbb Z$ and $c \in (0,1)$, let $H(R,r,c, k^*) $  be the event that there exists a spatial crossing of a rectangle $I\times J  \subset R $ so that  \\
(i) interval $I$ has length $\lfloor 2^{r \beta}\rfloor $ and its left endpoint is in $ \lfloor 2^{(r-k^*) \beta}\rfloor\IZ$, \\
(ii) interval $J$ has length less than $c2^r $. \\
Then        
\begin{equation}
P(H(R,r,c, k^*)) \ \leq \ C(c)  \frac{b-a}{2^{(r-k^*) \beta}} \frac{t-s}{2^r} P_r.
\end{equation}
\end{lemma}
\begin{proof}
Again we consider events
$$
A(i,j)= \{ \exists \text{ a spatial crossing of } [i \lfloor 2^{(r-k^*) \beta}\rfloor, i  \lfloor 2^{(r-k^*) \beta}\rfloor +  \lfloor 2^{r\beta}\rfloor ] \times [t_j, t_j + 2^ r]\},
$$
where $[i \lfloor 2^{(r-k^*) \beta}\rfloor, i  \lfloor 2^{(r-k^*) \beta}\rfloor +  \lfloor 2^{r\beta}\rfloor ] \subset [a,b] $ and
$t_j :=  s + 2^r (1-c) /2 \in [s,t]$.
Once more  $P(A(i,j) ) \ \leq \ P_r $ for all $(i,j)$ and the event $H(R,r,c, k^*) \ \subset \ \cup_{i,j} A(i,j)$ where the union is over $(i, j) $ satisfying the above constraint.
The number of such $(i,j) $ is the product of $ \lceil  (b-a) /  ( \lfloor 2^{(r-k^*) \beta}\rfloor) \rceil $ with $ \lceil  2(t-s)/2^r (1-c) \rceil $. By our assumptions
$b-a \geq 2^{(r-k^*) \beta } $ and $ t-s \geq \ 2^r $ this product is less than $4  \frac{b-a}{2^{(r-k^*) \beta}}  \times 8\frac {(t-s)}{2^r (1-c)} $.
\end{proof}

\vspace{0.3cm}

\noindent  {\bf Definition.}  For integer $\epsilon'>0$, $L\in\epsilon'\mathbb N$,
$T>0$  and space-time rectangle $D=[a, b] \times[0, T']$, with $T'\geq 3T$,
let $F ( \epsilon', L,T, D)$ be the event that there exists spatial interval $I' = [a',b'] \ \subset \ [a,b]$ and $[t_1, t_2] \ \subset \ [0, T']$ and a spatial crossing
of $I' \times [t_1,t_2]$, $\gamma : \ [t_1, t_2]  \subset [0,T'] \ \rightarrow \ I'$ so that \\
\indent
(i) $a',b' \ \in \ \epsilon ' \mathbb Z , \ b'-a' \  \leq  \ L$ \\
\indent
(ii) $t_2 - t_1  \in [T/2, 3T] $\\
\indent
(iii)$  \gamma(t_1) \ = \ a' , \gamma (t_2) \ = \ b' $.
\vspace{0.3cm}
\begin{prop}\label{p11} For $\epsilon ' < b-a $, there is a universal nontrivial $C$ so that  \label{diagonal}
\begin{eqnarray*}
&P(F ( \epsilon', L,T, D))\leq&\\
& C\left(\frac{T'}{T} \right)%
\frac{b-a}{\epsilon'}\frac{L}{\epsilon'}
P \left (\exists \textrm{ temporal crossing of } [0,L]\times [T', T'+3T]\right)^{\frac{1}{10}}.&
\end{eqnarray*}
\end{prop}
\vspace{0.3cm}

\begin{proof}
\noindent We choose $\epsilon = \frac{T}{17}$ and note that the event $F ( \epsilon', L, T, D)$ is contained in the union of
$$\left\{ \exists \textrm{ spatial crossing from}\;  \{k \epsilon'\}\times [i \epsilon, (i +1) \epsilon] \textrm{ to}\; \{(k+k') \epsilon'\}\times [j \epsilon, (j+1) \epsilon  ] \right\}$$
 over integers $i, j, k, k'$ relevant i.e. $ i\epsilon, (j+1) \epsilon \in [0, T']$, $(j-i) \epsilon \in [\frac{1}{2} T, 3T]$, $k \epsilon', (k+k')\epsilon' \in [a,b] \cap \epsilon ' \mathbb Z$.
By Corollary $\ref{FKG2}$ (see also Remark \ref{fkg}) 
the probability of this event is less than
 $$P\left( \exists \textrm{ temporal crossing of }[k \epsilon',k \epsilon' +k' \epsilon']\times [(i+1)\epsilon ,(i+1)\epsilon  + 3 T]\right)^{1/10},$$
 which is less than
 $$P\left( \exists \textrm{ temporal crossing of }[0,L]\times [(i+1)\epsilon ,(i+1)\epsilon  + 3 T]\right)^{1/10}.$$
 by monotonicity.  By our choice of $\epsilon, \, (i+1)\epsilon   <  T'$, so by the stochastic monotonicity of our renewal processes as used in the proof of Lemma \ref{build} (see Remark \ref{comparison}), this last term is dominated by
 $$P\left( \exists \textrm{ temporal crossing of }[0,L]\times [T' ,T' + 3 T]\right)^{1/10}.$$
 and the result follows from counting the number of choices of $k, k' $ as before.

\end{proof}



\subsection{Temporal crossings of $\lfloor 2^{n \beta}\rfloor\times 2^{n-k}$ rectangles}
\label{ssec;5.2}




We now apply the above estimates to the event of a temporal crossing of a $\lfloor 2^{n \beta}\rfloor\times 2^{n-k}$ rectangle,
where $k$ is a large fixed integer.
The goal is to prove
%

\begin{prop}\label{xiv}
	\noindent Let $k$ be a positive integer. For $1 \le i \le 2^k-1$, consider a collection
	$\{\tau_x,\,x\in[0, \lfloor 2^{n \beta}\rfloor]\}$ of time points in $[(i-1)2^{n-k}, i2^{n-k}]$, and a probability which is the the product of the infection Poisson process probability and the renewal probability on the timelines of $[0, \lfloor 2^{n \beta}\rfloor]$  starting from $\{(x,\tau_x),\,x\in[0, \lfloor 2^{n \beta}\rfloor]\}$. Let us call that probability $\tilde P$.
	Then there exists $n_0$ so that for $n \geq n_0$, the $\tilde P$-probability that there is a temporal crossing of $[0, \lfloor 2^{n \beta}\rfloor]\times [i2^{n-k}, (i+1)2^{n-k}]$
	is less than
	\begin{equation}\label{eq:xiv}
	C (k) \left(P_{n-k} \vee P_{n-k-1}\right)^{\frac{1}{10}},	
	\end{equation}
	uniformly over $\{\tau_x\}$, with $P_r$ as in Definition \ref{pr} and $C(k) $ a finite constant.
\end{prop}

\noindent {\bf Remark.} The situation described in the statement above comes up when we observe that a temporal crossing of $[0, \lfloor 2^{n \beta}\rfloor]\times [0, 2^{n}]$ implies $2^k$ temporal crossings of $\lfloor 2^{n \beta}\rfloor\times 2^{n-k}$ subrectangles. Taking advantage of the fact that $\int t^\alpha \mu(dt) <\infty$ for some $\alpha> 1$, we will (outside a set of small probability) restrict to crossings of $2^{k-1}$  alternating subrectangles, with given renewal starting marks in the timelines of previous respective subrectangles, to ensure that we can control the probabilities occurring in the recursion step of the proof. (See Subsection \ref {ssec;5.4}.)



\vspace{0.2cm}

Indeed consider a temporal crossing $(X(s))_{0 \leq s \leq 2^{n}}$ of $[0, \lfloor 2^{n \beta}\rfloor]\times [0, 2^{n}]$,
and  for $k$ large (but not depending on $n$) let us consider its restriction to the time interval
$[i2^{n-k}, (i+1) 2^{n-k}]$: $X_{k,i}=(X(s)\colon i 2^{n-k} \leq s \leq (i+1) 2^{n-k})$. We wish to show that there must be crossings of smaller rectangles of similar ``scale", yielding a probability estimate in terms of $P_{n-k}$.  Thus the above result accomplishes the second step of our strategy, as outlined at the end of the introduction.

\vspace{0.2cm}

{\it Proof of Proposition \ref{xiv}.}
We begin by breaking the latter kind of event into several cases.
Take $k_{0}$ so that $2^{-k_{0} \beta} \leq \frac{1-2^{- \beta}}{10}$ and $k_{0} > 7$.  We note that $k_0$, once fixed, does not depend on $n$.
%
%
We split up the argument into three cases. For this let $v(X_{k,i})$ be as in \eqref{var}.

\vspace{0.2cm}

\noindent Case 0. $v(X_{k,i}) > (1+ \frac{ 2^{-k_0 \beta }}{4} )  \lfloor 2^{\beta (n-k)} \rfloor$.

\vspace{0.3cm}

\noindent Case 1. $ v(X_{k,i})< \lfloor 2^{(n-k) \beta} \rfloor  \left( 1- \frac{(1 - 2^{-\beta})}{10}\right)$.

\vspace{0.3cm}

\noindent Case 2.
There exist $\tau_{i} < \sigma_{i} \in [i 2^{n-k} ,(i+1) 2^{n-k}]$ with $\sigma_{i} - \tau_{i} < \frac{9}{20}2^{n-k} $ and:


$(i)\, \lfloor 2^{(n-k) \beta} \rfloor  \left( 1- \frac{(1 - 2^{-\beta})}{10}\right) \le \vert X(\sigma_{i}) - X (\tau_{i}) \vert  \leq   (1+ 2^{-k_0 \beta } ) 2^{(n-k) \beta} $;

\vspace{0.3cm}

$(ii)\, (X(s) - X(\sigma_{i})) (X(s) - X(\tau_{i})) \leq 0 \hspace{0.3cm} \text{ for all } s \in  [i 2^{n-k} ,  (i+1) 2^{n-k}]$.

\vspace{0.3cm}

\noindent Case 3. As in Case 2, but instead  $\sigma_{i} - \tau_{i} \ge \frac{9}{20}\,2^{n-k} $.

\vspace{0.3cm}

\noindent The probability of the event in Case 0 is dealt with by Lemma \ref{lemad2} with $c = (1 + 2^{-k_0\beta} / 4 ) ^{-1}$. It is bounded by a constant times $2^{k\beta }P_{n-k}$.

\vspace{0.3cm}

\noindent Case 1 implies the occurrence of the event $B_{t} (c, [0, 2^{n \beta}],n-k)$ for $t=i2^{n-k}$, $c = 1 - \frac{1-2^{-\beta}}{10}$.  Note that given the FKG property of the renewal processes (see Remark \ref{comparison}) and the fact that event $B_{t} (c, [0, 2^{n \beta}],n-k) $ is a decreasing event for the renewal points,
the probability of $B_{t} (c, [0, 2^{n \beta}],n-k) $ under $\tilde P$ is bounded from above by the probability of $B_{t'} (c, [0, 2^{n \beta}],n-k) $ under $P$, with $t' \ = \ 2^{n-k}$. By Lemma \ref{lem4} its probability is bounded by $C(c) 2^{k \beta} P_{n-k}$ for suitable finite $C(c)$.

\vspace{0.3cm}
\noindent In Case 2, since $\sigma_i - \tau_i < \frac{9}{20}2^{n-k } $, the event $A (D, c, n-k-1)$ occurs for
$D = [0, \lfloor2^{n \beta}\rfloor]\times [i 2^{n-k}, (i+1) 2^{n-k}]$ and
$1/c = \min \left( \frac{10}{9} , \frac{1}{10}+ \frac{9}{10} 2^{\beta} \right) $.  Again, as in Case 1, under the probability $\tilde P $  this probability is bounded  $P ( A (D', c, n-k-1)) $ where $D' \ = \ [0, \lfloor2^{n \beta}\rfloor]\times [ 2^{n-k}, 2 \  2^{n-k}]$.   So by Proposition \ref{adcr}, this is bounded by a multiple of $P_{n-k-1}$.

\vspace{0.3cm}

\noindent In Case 3, retaining the notation introduced in Case 2, we assume without loss of generality
that $X( \tau_i ) < X( \sigma_i )$ and define
\begin{eqnarray*}
\tau_i' \!\!&=&\!\!  \inf \{s \geq \tau_i : X(s) \geq X( \tau_i ) + \lfloor 2^{(n-k-k_0) \beta }\rfloor, X(s) \in \lfloor 2^{(n-k-k_0) \beta }\rfloor \IZ\};\\
\tau_i'' \!\!&=&\!\! \sup \{\tau_i'\leq s \leq \sigma_i : X(s) = X( \tau_i' )\};
\end{eqnarray*}
and (symmetrically)
\begin{eqnarray*}
	\sigma_i' \!\!&=&\!\!  \sup \{s \leq \sigma_i : X(s) \leq X( \sigma_i ) - \lfloor 2^{(n-k-k_0) \beta }\rfloor, X(s) \in \lfloor 2^{(n-k-k_0) \beta }\rfloor \IZ\};\\
	\sigma_i'' \!\!&=&\!\! \inf \{\tau_i\leq s \leq \sigma_i' : X(s) = X( \sigma_i' )\}.
\end{eqnarray*}


We
%
%
%
%
%
%
%
%
%
have two subcases, depending on $\sigma_i''-\tau_i''$:
\begin{enumerate}
	\item If $\sigma_i''-\tau_i''\leq\frac342^{n-k-1}$, then
letting  $D = [0, \lfloor 2^{n \beta}\rfloor]\times [ i 2^{n-k}, (i + 1) 2^{n-k}]$,
we claim that the event $H(D,r,c,k^*) $ has occurred with 
$c = 3/4, r = n-k-1$ and $k^* = k_0$. Indeed the path from $\tau''_i$ to $\sigma''_i$ ensures it,
since $|X(\tau''_i)-X(\sigma''_i)|=|X(\tau'_i)-X(\sigma'_i)|\geq |X(\sigma_i)-X(\tau_i)|-4\times2^{(n-k-k_0)\beta}\geq  \lfloor 2^{r\beta}  \rfloor $,
%
%
where we use the lower bound in Case 2 $(i)$ and the first condition on $k_0$ stipulated above and for $n$ large we have $\frac{ \lfloor 2^{(n-k) \beta} \rfloor}{ \lfloor 2^{(n-k-1) \beta} \rfloor} $
is approximately $2^{\beta} $. From Lemma~\ref{lemad1}, after suitably shifting the time domain as before, we get a $\tilde P$ probability bound of constant times $P_{n-k-1}$ for this subcase, where the constant depends on $k,k_0$ but not on $n$.	

	\item If $\sigma_i''-\tau_i''>\frac342^{n-k-1}$, then the path between $\tau''_i$ and $\sigma''_i$ implies the occurrence of
	$F (\epsilon', L, T, D)$ for the same $D$ as above, and
	$$\epsilon' = \lfloor 2^{(n-k-k_{0}) \beta}\rfloor,\quad
	T =\frac 13  2^{n-k}, \quad L=\lfloor2^{(n-k)\beta}\rfloor.$$
	From Proposition~\ref{p11}, we get a $\tilde P $  probability bound of constant times $P^{\frac1{10}}_{n-k}$ for this subcase, where again the constant depends on $k,k_0$ but not on $n$.	

\end{enumerate}

Collecting these cases together we have that one of the above four cases must occur given our crossing and that the probability of each of them has a bound of the form demanded.  The proof is complete.
%





\subsection{Spatial crossings of $\lfloor 2^{(n-k) \beta}\rfloor\times 2^{n}$ rectangles}
\label{ssec;5.3}
In this subsection we derive a bound similar to~(\ref{eq:xiv}) for spatial crossings of
$\lfloor 2^{(n-k) \beta}\rfloor\times 2^{n}$ rectangles, with $k$ a fixed number (to be chosen later).
This case allows for a more direct, simpler analysis than the one employed in the previous two subsections.


Let us fix $k\leq n$ and consider $D:=[0, \lfloor 2^{(n-k)\beta}\rfloor]\times [0, 2^{n}]$, which may be written as $\cup_{i=1}^{2^k}D_i$, with $D_i:=[0, \lfloor 2^{(n-k)\beta}\rfloor]\times [(i-1)2^{n-k}, i2^{n-k}]$.
Let now $R_i$ denote the event that there exists a  spatial crossing of $D$ starting on the left hand side of $D_i$. $R_i$  may be partitioned into  $R_i^{\mbox{\tiny $\rightarrow$}}$, $R_i^{\mbox{\tiny $\nearrow$}}$ and $R_i^{\mbox{\tiny $\uparrow$}}$, meaning that the crossing ends on the right hand side of $D_i$, $D_{i+1}$, and $D_{j}$ for some $j>i+1$, respectively. The probabilities of the first and third events are bounded above by $P_{n-k}$, since they imply a spatial crossing of $D_i$ and a temporal crossing of $D_{i+1}$, respectively.

To bound the probability of $R_i^{\mbox{\tiny $\nearrow$}}$, we partition this event as follows.
Let $D_i^-:=[0, \lfloor 2^{(n-k)\beta}\rfloor]\times [(i-1)2^{n-k}, (i-\frac12)2^{n-k}]$ and
$D_i^+:=[0, \lfloor 2^{(n-k)\beta}\rfloor]\times [(i-\frac12)2^{n-k}, i2^{n-k}]$, and similarly define
$D_{i+1}^-$ and $D_{i+1}^+$. We then partition  $R_i^{\mbox{\tiny $\nearrow$}}$ into
$R_{i,i+1}^{\mbox{\tiny $\rightarrow$}}$, $R_{i,i+1}^{\mbox{\tiny $\uparrow$}}$, $R_{i,i+1}^{\mbox{\tiny $\nearrow$}}$,
and $\tilde R_{i,i+1}^{\mbox{\tiny $\nearrow$}}$, where the crossing
starts on the left of $D_i^+$ and ends on the  right of $D_{i+1}^-$,
starts on the left of $D_i^-$ and ends on the  right of $D_{i+1}^+$,
starts on the left of $D_i^-$ and ends on the right of $D_{i+1}^-$,
starts on the left of $D_i^+$ and ends on the right of $D_{i+1}^+$,
respectively.

The probabilities of the first and second events are bounded above by $P_{n-k}$, since they imply a spatial crossing of
$D_i^+\cup D_{i+1}^-$,
and a temporal crossing of the same rectangle, respectively.

Let us now bound $P(R_{i,i+1}^{\mbox{\tiny $\nearrow$}})$. Let $\tilde R_{i,i+1}^{\mbox{\tiny $\nwarrow$}}$ denote the event that there exists a spatial crossing of $D$ starting on the {\it left} hand side of $D_i^+$ and ending on the right hand side of $D_{i+1}^+$. Since the event where there is a temporal crossing of $D_i^+\cup D_{i+1}^-$ contains
$R_{i,i+1}^{\mbox{\tiny $\nearrow$}}\cap R_{i,i+1}^{\mbox{\tiny $\nwarrow$}}$, we find, arguing similarly as in the proof of Lemma~\ref{build}, that the probability of the former event bounds from above  $P(R_{i,i+1}^{\mbox{\tiny $\nearrow$}})^2$, and thus
$$ P(R_{i,i+1}^{\mbox{\tiny $\nearrow$}})\leq P_{n-k}^{1/2}.$$

We may similarly obtain the same bound for $P(\tilde R_{i,i+1}^{\mbox{\tiny $\nearrow$}})$.

Collecting all the above bounds, we get that
\begin{equation}\label{space}
P(R)\leq C 2^k P_{n-k}^{1/2},
\end{equation}
where $R=\cup_{i=1}^{2^k}R_i$ is the event that there exists a spatial crossing of $D$ starting on its left hand side.


\subsection{Proof of Proposition \ref{propkey} --- Recursion}
\label{ssec;5.4}




We now use the previous estimates to set up a recursion for $P_n$ --- see~(\ref{recur}) below ---, which readily leads to the conclusion of our proof of Proposition~\ref{propkey}, as subsequently explained, thus fulfilling the third step of our strategy, as outlined at the end of the Introduction.

%


\noindent

Consider first the probability of a temporal crossing of space-time rectangle $ [0, \lfloor 2^{n \beta}\rfloor]\times [0, 2^{n}]$
where no point in $[0, \lfloor 2^{n \beta}\rfloor]$ has a $2^{n-k}$ long interval in its timeline between times $-2^{n-k}$ and $2^{n} +2^{n-k}$ with no renewal marks in it; we speak of a $2^{n-k}$-gap in $[-2^{n-k}, 2^{n} +2^{n-k}]$ in this context.
We can analyse the probability of a temporal crossing of $[0, \lfloor2^{n \beta}\rfloor] \times [0, 2^{n}]$ via the filtration of the
Poisson processes/renewal processes.

\noindent

More specifically we define $\mathcal{G}_{2i}$ as the $\sigma$-field generated by these processes for all 
$x \in [0, 2^{n \beta}]$ up to time $2 i 2^{n-k}$, while $\mathcal{G}_{2i +1}$ is the $\sigma$-field generated by $\mathcal{G}_{2i}$ plus random variables $V_{x}^{2i +1}= \inf \{t \geq 2i2^{n-k}$: $t$ is in $\mathcal{R}_x \}$.
We put $T_n = \inf \{2i +1 : \exists x \in [0, 2^{n \beta}] \hspace{0.3cm} V_{x}^{2i +1} \geq (2i +1) 2^{n-k} \}$. 
$T_n$ is a stopping time for this filtration and
\begin{equation}\label{bound}
P ( T_n\leq 2^{k}) \leq K 2^{-n (\alpha -1 - \beta)} \equiv K 2^{-n \epsilon_{0}},
\end{equation}
for some $K$ depending only on $k$.


For $i=1,\ldots, 2^k$, let $G_i$ denote the event that there exists a temporal crossing of the rectangle
$[0, \lfloor 2^{n \beta}\rfloor] \times [i2^{n-k}, (i+1)2^{n-k}]$, and let $J_i$ denote the event that there
is no $2^{n-k}$-gap in $[0, \lfloor 2^{n \beta}\rfloor] \times [i2^{n-k}, (i+1)2^{n-k}]$.



We then have 
\begin{eqnarray}\nonumber
& P \left(\exists \text{ a temporal crossing of }   [0, \lfloor 2^{n \beta}\rfloor] \times [0, 2^{n}]\right) &\\\label{rec}
&\leq  P ( T_n \leq 2^{k})+P(G_2)\prod_{j=2}^{2^{k-1}} P(G_{2j}|G_2,\ldots,G_{2(j-1)},J_{2j-1}).&
\end{eqnarray}

The probabilities inside the product on the right hand side of~(\ref{rec}) can be written in terms of an integral over
conditional probabilities of $G_{2j}$ given renewal histories up to the first renewal mark (in chronological order)
in each time line contained in $[0, \lfloor 2^{n \beta}\rfloor] \times [(2j-1)\,2^{n-k}, 2j\,2^{n-k}]$
--- let us denote such renewal mark at the time line of $x\in[0, \lfloor 2^{n \beta}\rfloor]$ by $(x,\tau^j_x)$ ---
and Poissonian infection histories up to time $(2j-1)\,2^{n-k}$.
Actually, that conditional probability equals
\begin{equation}\label{condprob}
P\left(G_2\left|\mbox{first renewal marks} = \{(x,\tau^j_x-(2j-1)\,2^{n-k}),\,
x\in [0, \lfloor 2^{n \beta}\rfloor]\}\right.\right).
\end{equation}
Notice that the conditioning first renewal marks belong to timelines in
$[0, \lfloor 2^{n \beta}\rfloor] \times [0, 2^{n-k}]$.
One now has that each one of these conditional probabilities satisfies the conditions of Proposition~\ref{xiv}, and so are (uniformly) bounded by the expression in~(\ref{eq:xiv}), and thus so is the integral, and clearly also $P(G_2)$. It follows that the right hand side of~(\ref{rec}) is bounded above by
\begin{eqnarray}\nonumber
&&P ( T_n \leq 2^{k})+  C (\,k)\left(P_{n-k-1}^{\frac{1}{10}} \vee P_{n-k}^{\frac{1}{10}} \right)^{2^{k-1}}\\
&\leq & P ( T_n \leq 2^{k})+  C' (\, k) \left(P_{n-k-1} \vee P_{n-k} \right)^{2},
\end{eqnarray}
if $2^{k-1} > 20$.

\begin{remark}\label{gap}	
If we had a gap in $[0, \lfloor 2^{n \beta}\rfloor] \times [(2j-1)\,2^{n-k}, 2j\,2^{n-k}]$, say in the
timeline of $x\in[0, \lfloor 2^{n \beta}\rfloor]$,
then we would know that $\{x\}\times[(2j-1)\,2^{n-k}, 2j\,2^{n-k}]$ had no renewal mark, and the corresponding conditional probability would not be a renewal probability measure with interarrival distribution $\mu$ starting at a given time, as prescribed in Definition~\ref{pr}. We would not have a bound in terms of $P_{\cdot}$.\\
\indent We note also that the alternating of $G_\cdot$ and $J_\cdot$ events in~(\ref{rec}) allows for the validity of~(\ref{condprob}), enabling the comparison to $P_{\cdot}$; on the other hand, we get the power of $2^{k-1}$ which boosts the power of $\frac1{10}$ to $2$.
\end{remark}



The estimation of the probability of a spatial crossing of a space-time rectangle
$ [0, \lfloor 2^{n \beta}\rfloor]\times [0, 2^{n}]$ is similar, if easier.
A spatial crossing of that rectangle starting from its left hand side entails $\lfloor 2^{k \beta}\rfloor$
crossings of $\lfloor 2^{(n-k) \beta}\rfloor\times 2^{n}$ rectangles starting from their respective left hand sides,
which is a collection of independent events, each of whose probabilities is bounded above by the right hand side
of~(\ref{space}), as argued in Subsection~\ref{ssec;5.3} above.
Of course, the probability of the event of a spatial crossing starting on the right hand side of
$ [0, \lfloor 2^{n \beta}\rfloor]\times [0, 2^{n}]$ satisfies the same bound.


 We thus have that if $2^{(k-1)\beta} > 4$,
 $$
 P ( \exists \mbox{ a spatial crossing of } [0, \lfloor 2^{n \beta}\rfloor] \times [0, 2^{n}] )
 \leq C(k)\,P_{n-k}^{2},
 $$
 for some $C(k)$ not depending on $n$.

Thus we can find $k$ so that for all $n$ large
\begin{equation}\label{recur}
P_n\leq P ( T_n \leq 2^{k}) + C''  \left(P_{n-k-1} \vee P_{n-k} \right)^{2},
\end{equation}
%
%
where $C''$ depends only on $k$.
Here $P_n$
represents the  supremum over renewal probabilities on $[0, 2^{n \beta}]\times [0, 2^{n}]$ as in Definition~\ref{pr}.

\medskip

To complete the proof of Proposition \ref{propkey} we note that it follows from~(\ref{bound}) that if $n$ is large, then
$P ( T_n \leq 2^{k}) \leq 2^{-n \frac{\epsilon_{0}}{2}}$.  Furthermore,
for $n_0$ an integer fixed large and $j$ a strictly positive integer, let $\mathcal{H}(j) $
be the statement
\begin{equation}\label{star}
P_r \leq 2^{-r\frac{\epsilon_{0}}{5}}\mbox{ for each }n_0 \leq r \leq n_0 +j(k+1).
\end{equation}

If $\mathcal{H}(j) $ holds, then
applying (\ref{recur}), $P_n \leq 2^{-n \frac{\epsilon_{0}}{2}} + C'' \left(P_{n-k-1} \vee P_{n-k} \right)^{2} $.  Under $\mathcal{H}(j) $ this is less than
$2^{-n \frac{\epsilon_{0}}{2}} + C'' 2^{-2(n-k-1) \frac{\epsilon_{0}}{5}}$.
If $n_0 $ was fixed sufficiently large this is
$\leq 2^{-n \frac{\epsilon_{0}}{5}}$ for $n= n_0 +j(k+1) +1,n _0 +j(k+1) +2, \dots, n_0 +(j+1)(k+1)-1 $. We can now apply this argument again for $n = n_0 +(j+1)(k+1)$ and we have established the inductive hypothesis that $\mathcal{H}(j) $ implies $\mathcal{H}(j+1) $ ;
if necessary making $n_0$ larger, we further have that
$P ( T_{n_0+i} \leq 2^{k}) \leq 2^{-(n_{0}+i) \frac{\epsilon_{0}}{2}}$ for $0 \le i \leq k$.
We now choose $\lambda_0$ so small that~(\ref{star}) holds  for $j=1$ and $\lambda \in (0, \lambda_0)$.

\vspace{0.3cm}

\noindent {\textbf {Acknowledgements:}}  L.~R.~Fontes and M.~E.~Vares thank CBPF for the hospitality in the week January 15-20, 2018. L.~R.~Fontes acknowledges support of CNPq (grant 311257/2014-3) and FAPESP (grant 2017/10555-0). M.~E.~Vares acknowledges support of CNPq (grant 305075/2016-0) and FAPERJ (grant E-26/203.048/2016).

\vspace{0.3cm}
%

%
%
%
%
%
%
%
%

\end{document}